\documentclass[11pt,letterpaper, reqno]{amsart}
\newcommand{\la}{\langle}
\newcommand{\ra}{\rangle}
\renewcommand{\Re}{\operatorname{Re}}
\renewcommand{\Im}{\operatorname{Im}}

\newcommand{\sech}{\operatorname{sech}}

\newcommand{\supp}{\operatorname{supp}}

\usepackage{amssymb,color}
\usepackage{amsthm}
\usepackage{graphicx}
\usepackage{mathabx}

\newtheorem{theorem}{Theorem}

\newtheorem{proposition}[theorem]{Proposition}
\newtheorem{lemma}[theorem]{Lemma}

\theoremstyle{remark}
\newtheorem{remark}[theorem]{Remark}

\numberwithin{equation}{section}

\numberwithin{theorem}{section}

\numberwithin{table}{section}

\numberwithin{figure}{section}

\ifx\pdfoutput\undefined
  \DeclareGraphicsExtensions{.pstex, .eps}
\else
  \ifx\pdfoutput\relax
    \DeclareGraphicsExtensions{.pstex, .eps}
  \else
    \ifnum\pdfoutput>0
      \DeclareGraphicsExtensions{.pdf}
    \else
      \DeclareGraphicsExtensions{.pstex, .eps}
    \fi
  \fi
\fi

\title[Traveling solitons for fractional NLS]{A new class of Traveling Solitons for cubic Fractional Nonlinear Schr\"odinger equations}

\date{\today}
\author{Younghun Hong}
\address{University of Texas, Austin}
\author{Yannick Sire}
\address{Universit\'e Aix-Marseille, I2M, UMR 7353, Marseille, France}

\begin{document}

\maketitle

\begin{abstract}
We consider the one-dimensional cubic fractional nonlinear Schr\"odinger equation 
$$i\partial_tu-(-\Delta)^\sigma u+|u|^{2}u=0,$$
where $\sigma \in (\frac12,1)$ and the operator $(-\Delta)^\sigma$ is the fractional Laplacian of symbol $|\xi|^{2\sigma}$. Despite the lack of any Galilean-type invariance, we construct a new class of traveling soliton solutions of the form
$$u(t,x)=e^{-it(|k|^{2\sigma}-\omega^{2\sigma})}Q_{\omega,k}(x-2t\sigma|k|^{2\sigma-2}k),\quad k\in\mathbb{R},\ \omega>0$$
by a rather involved variational argument. 
\end{abstract}

\tableofcontents

\section{Introduction}
This paper is devoted to the investigation of the one-dimensional cubic focusing fractional nonlinear Schr\"odinger equations (NLS)
\begin{equation}\tag{$\textup{NLS}_\sigma$}
i\partial_tu-(-\Delta)^\sigma u+|u|^{2}u=0,
\end{equation}
where $\sigma\in(\frac{1}{2},1)$. The fractional Schr\"odinger equations have been introduced by Laskin \cite{laskin} in the context of quantum mechanics, generalizing the Feynman integral using Levy processes (see \cite{B} for an account on Levy processes). Moreover, the fractional NLS has been investigated in \cite{ozawa,HS,COX,zihua,CHKL,FL,FLS2}  (see also references therein). 

The goal of the present work is to construct a family of traveling solitons for the fractional NLS.

In the Laplacian case ($\sigma=1$), traveling solitons can be formed easily by boosting static solitons by the Galilean transformation. Precisely, a static soliton for the cubic NLS is given by $e^{it\omega^2} Q_\omega(x)$ with $\omega>0$, where $Q_\omega(x)=\sqrt{2}\omega\sech (\omega x)$ is the ground state for the nonlinear elliptic equation
$$-\Delta u+\omega^2 u-|u|^2 u=0.$$
Then, since the equation is invariant under the Galilean transformation
$$u(t,x)\mapsto e^{-it|k|^2}e^{ik\cdot x}u(t,x-2tk),\quad k\in\mathbb{R},$$
boosting a static solution, we obtain a traveling soliton
$$e^{-it(|k|^{2}-\omega^{2})} e^{ik\cdot x}Q_\omega(x-2tk)$$
with a velocity of $2k$.

In the fractional case, the ground state $Q_\omega$ for the elliptic equation 
\begin{equation}\label{fractional elliptic equation}
(-\Delta)^\sigma u+\omega^2 u-|u|^2 u=0
\end{equation}
is known to exist when $\frac{1}{4}<\sigma<1$ \cite[Proposition 1.1]{FL}, and thus $e^{it\omega^{2\sigma}}Q_\omega(x)$ is a static soliton solution to the fractional NLS. However, contrary to the Laplacian case, one cannot construct a traveling soliton simply by boosting a static soliton, since the equation does not have any exact Galilean-type invariance due to the non-locality of the fractional Laplacian $(-\Delta)^\sigma$.

Nevertheless, as observed in \cite{HS}, the fractional NLS is {\sl almost invariant} under the pseudo-Galilean transformation 
\begin{equation}\label{pseudoGal}
\mathcal G_k: u(t,x)\mapsto e^{-it|k|^{2\sigma}}e^{ik\cdot x}u(t,x-2t\sigma|k|^{2(\sigma-1)}k)
\end{equation}
for smooth solutions. 

By almost invariance we mean: if $u(t)$ solves $(\textup{NLS}_\sigma)$, then $\tilde{u}(t)=\mathcal{G}_ku(t)$ obeys $(\textup{NLS}_\sigma)$ with an error term
\begin{equation}\label{fNLS with error}
i\partial_t\tilde{u}+(-\Delta)^\sigma\tilde{u}+\omega|\tilde{u}|^{p-1}\tilde{u}=e^{it|v|^{2\sigma}}e^{-ik\cdot x}(\mathcal{E}u)(t,x-2\sigma t|k|^{2(\sigma-1)}k),
\end{equation}
where
$$\widehat{\mathcal{E}u}(\xi)=E(\xi)\hat{u}(\xi)$$
with 
$$E(\xi)=|\xi-v|^{2\sigma}-|\xi|^{2\sigma}-|v|^{2\sigma}+2\sigma|v|^{2(\sigma-1)}v\cdot\xi.$$

Therefore, it is still natural to consider an ansatz of the form
\begin{equation}\label{ansatz}
u_{\omega,k}(t,x)=e^{-it(|k|^{2\sigma}-\omega^{2\sigma})} e^{ik\cdot x}Q_{\omega,k}(x-2t\sigma|k|^{2\sigma-2}k),
\end{equation}
which will lead to a natural family of moving solitons with frequency $\omega$ and speed $k$. The profile $Q_{\omega,k}$ then solves the pseudo-differential equation 
\begin{equation}\label{Q equation}
\mathcal{P}_k Q_{\omega,k}+\omega^{2\sigma}Q_{\omega,k}-|Q_{\omega,k}|^{2}Q_{\omega,k}=0,
\end{equation}
where
$$\mathcal{P}_k=e^{-ik\,x}(-\Delta)^\sigma e^{ik\,x}-|k|^{2\sigma}+2i\sigma|k|^{2\sigma-2}k\cdot\nabla_x.$$

In this paper, we establish existence of a solution $Q_{\omega,k}$ to the equation \eqref{Q equation} and the traveling soliton $e^{it\omega^{2\sigma}}Q_{\omega,k}(x)$ for the fractional NLS.
\begin{theorem}\label{existence} Let $\sigma\in(\frac{1}{2},1)$. For any $k \in \mathbb R$, there exists $Q_{\omega,k} \in H^1(\mathbb R)$ solving \eqref{Q equation} for some $\omega>0$. Furthermore, we have $Q_{\omega,k} \in C^\infty(\mathbb R)$.
\end{theorem}

\begin{remark}
In the book \cite{sulem}, an exact solution of the standard NLS equation is also found. However, here in our context, we consider more general equations, i.e. nonlocal ones, and our solutions are also exact but differ substantially from the ones of Sulem and Sulem. 
\end{remark}

The key observation to prove the theorem is that the pseudo-differential operator $\mathcal{P}_k$ is an elliptic operator. Indeed, $\mathcal{P}_k$ is a Fourier multiplier, $\widehat{\mathcal{P}_k f}(\xi)=p_k(\xi)\hat{f}(\xi)$, with the symbol
$$p_k(\xi)=|\xi+k|^{2\sigma}-|k|^{2\sigma}-2\sigma|k|^{2\sigma-2}k\cdot\xi.$$
So, we have $p_k(0)=0$. Differentiating $p_k(\xi)$, we get or $k\neq 0$
\begin{equation}\label{p_k}
\begin{aligned}
p_k'(\xi)&=2\sigma |\xi+k|^{2\sigma-2}(\xi+k)-2\sigma|k|^{2\sigma-2}k&&\Rightarrow &&p_k'(0)=0,\\
p_k''(\xi)&=2\sigma(2\sigma-1)|\xi+k|^{2\sigma-2} &&\Rightarrow && p_k''(\xi)\geq0.
\end{aligned}
\end{equation}
Thus, $p_k(\xi)$ is non-negative for all $\xi$ for $\sigma >1/2$. 

This fact allows us to apply the variational approach as in \cite{weinstein} for instance. It is however more involved, since the symbol of the operator $\mathcal P_k$ can be described differently according to low and high frequencies. Indeed, by \eqref{p_k}, we find the asymptotics of the symbol $p_k(\xi)$:
\begin{equation}\label{p_k asymptotics}
\begin{aligned}
p_k(\xi)&= |\xi|^{2\sigma}+O(|\xi|^{2\sigma-1}) &&\textup{as}&&\xi\to\infty,\\
p_k(\xi)&= \sigma(2\sigma-1)|k|^{2\sigma-2}|\xi|^2+O(|\xi|^3) &&\textup{as}&&\xi\to0.
\end{aligned}
\end{equation}
Therefore, the operator $\mathcal{P}_k$ behaves like $(-\Delta)^\sigma$ in high frequencies, and it behaves like $\sigma(2\sigma-1)|k|^{2\sigma-2}(-\Delta)$ in low frequencies.

Taking the inhomogeneity of the operator $\mathcal{P}_k$ into account, we will introduce a new Weinstein functional $\mathcal{W}$ (see \eqref{mathcal W}) associated with the Gagliardo-Nirenberg type inequality (Theorem \ref{GN for P_1}), and then we will achieve the existence of a profile $Q_{\omega,k}$ as a minimizer for the functional $\mathcal{W}$. 

\begin{remark}
Our proof does not allow to treat the case $\sigma \in (0,\frac12]$, since the operator $\mathcal P_k$ becomes degenerate. Indeed, the symbol of $\mathcal P_k$ is not positive definite. In particular, the second derivative of the symbol vanishes when $\sigma=\frac12$ (see \eqref{p_k asymptotics}). We are not aware even of numerical computations dealing with $\sigma \in (0,\frac12]$. 
\end{remark}

\begin{remark}
When $\sigma=\frac{1}{2}$, in \cite{krieger}, the authors constructed moving solitons using the ansatz 
\begin{equation}\label{ansatz2}
u(t,x)=e^{it}Q_v(x-vt),\quad |v|<1.
\end{equation}
It is easy to check that if $u(t)$ solves the fractional NLS, then the profile $Q_v$ solves a time-independent equation that is elliptic if $|v|<1$, while it is not elliptic if $|v|\geq 1$. As mentioned above, when $\sigma=\frac{1}{2}$, one cannot construct a traveling soliton using the ansatz \eqref{ansatz}. However, when $\sigma\in (\frac{1}{2},1)$,  the ansatz \eqref{ansatz2} is not appropriate, since in this case, $Q_v$ solves the non-elliptic equation. We also remark that contrary to the half-Laplacian case, if $\sigma\in (\frac{1}{2},1)$, one can construct traveling solitons at any speed. 
\end{remark}

We now describe the general strategy of the proof:
\begin{itemize}
\item We first derive a Gagliardo-Nirenberg inequality for a suitable operator. 
\item We then introduce the functional to be minimized. The Euler-Lagrange equation of this functional gives us precisely the equation under consideration. 
\item By means of a profile decomposition, we provide the existence of a minimizer for this functional. 
\end{itemize}

\begin{remark}
In the paper \cite{ABS}, a detailed study of some problems in fluid dynamics is performed. It happens that our result Theorem \ref{existence} is contained in this paper. Indeed, in \cite{ABS}, Theorem 3.1 provides the existence of traveling waves for a model involving a pseudo-differential operator. Following the remarks page 1248, our pseudo-differential operators satisfies the necessary conditions so that their theorem applies. In spite of this, we emphasize that our proof is different from the one of Albert {\sl et} and that it comes from a different context.  
\end{remark}

\section{Reduction to the Case $k=1$, Gagliardo-Nirenberg Inequality for $\mathcal{P}_1$ and Minimization Problem}

For notational conveniences, we drop the index $\omega$ in $Q_{\omega,k}$. Let $Q_k(x)=|k|^\sigma Q_1(kx)$. Then, by scaling, one has 
\begin{equation}\label{scaling}
(\mathcal{P}_k Q_k)(x)=|k|^{3\sigma}(\mathcal{P}_1Q_1)(kx)
\end{equation}
Indeed, taking the Fourier transform
\begin{align*}
\widehat{(\mathcal{P}_k Q_k)}(\xi)&=(|\xi+k|^{2\sigma}-|k|^{2\sigma}-2\sigma |k|^{2\sigma-2}k\cdot\xi) |k|^{\sigma-1}\hat{Q}_1(\tfrac{\xi}{k})\\
&=|k|^{3\sigma-1}(|\tfrac{\xi}{k}+1|^{2\sigma}-1-2\sigma\tfrac{\xi}{k})\hat{Q}_1(\tfrac{\xi}{k})\\
&=|k|^{3\sigma-1} (\mathcal{P}_1Q_1)^\wedge(\tfrac{\xi}{k}).
\end{align*}
Hence, $Q_1$ solves
$$\mathcal{P}_1Q_1+\omega^{2\sigma}Q_1-|Q_1|^2Q_1=0$$
if and only if $Q_k$ solves 
$$\Big\{\mathcal{P}_k Q_k+(|k|^{2\sigma}\omega^{2\sigma})Q_k-|Q_k|^2Q_k\Big\}(x)=|k|^{3\sigma}\Big\{\mathcal{P}_1Q_1+\omega^{2\sigma}Q_1-|Q_1|^2Q_1\Big\}(kx)=0.$$
Therefore, without loss of generality, we reduce our study to the case $k=1$.

Now we state the Gagliardo-Nirenberg inequality associated with the operator $\mathcal{P}_1$.

\begin{theorem}[Gagliardo-Nirenberg inequality for $\mathcal{P}_1$]\label{GN for P_1} For $\theta\in(0,1)$, we have 
\begin{align*}
\|u\|_{L^4}^4&\lesssim \|u\|_{L^2}^{\frac{4\sigma-1}{\sigma}}\|\mathcal{P}_1^{1/2} u\|_{L^2}^{\frac{1}{\sigma}}+\alpha\|u\|_{L^2}^3\|\mathcal{P}_1^{1/2} u\|_{L^2}\\
&-\Big\{\|u\|_{L^2}^{\frac{4\sigma-1}{\sigma}}\|\mathcal{P}_1^{1/2} u\|_{L^2}^{\frac{1}{\sigma}}\Big\}^{1-\theta}\Big\{\alpha\|u\|_{L^2}^3\|\mathcal{P}_1^{1/2} u\|_{L^2}\Big\}^\theta,
\end{align*}
where $\alpha=\alpha(\sigma)=1/\sqrt{\sigma(2\sigma-1)}.$
\end{theorem}

\begin{proof}
By Young's inequality, it suffices to show that
$$\|u\|_{L^4}^4\lesssim \|u\|_{L^2}^{\frac{4\sigma-1}{\sigma}}\|\mathcal{P}_1^{1/2} u\|_{L^2}^{\frac{1}{\sigma}}+\alpha \|u\|_{L^2}^3\|\mathcal{P}_1^{1/2} u\|_{L^2}.$$
Indeed, denoting 
$$
a=\|u\|_{L^2}^{\frac{4\sigma-1}{\sigma}}\|\mathcal{P}_1^{1/2} u\|_{L^2}^{\frac{1}{\sigma}},\,\,\,b=\alpha\|u\|_{L^2}^3\|\mathcal{P}_1^{1/2} u\|_{L^2},
$$
the inequality writes 
$$
\|u\|_{L^4}^4\lesssim a +b-a^{1-\theta}b^\theta. 
$$
If ones proves that $\|u\|_{L^4}^4\lesssim a +b$, then using Young's inequality $a^{1-\theta}b^\theta\leq (1-\theta)a+\theta b$, one gets 
$$
\|u\|_{L^4}^4\lesssim a +a b - a^{1-\theta}b^\theta +(1-\theta)a+\theta b.
$$
Therefore, relabeling the constants, this gives the desired result.

Let $\varphi\in C_c^\infty$ be a smooth cut-off such that $\varphi=1$ on $[-\frac{1}{4}, \frac{1}{4}]$ and $\supp\varphi\subset[-\frac{1}{2}, \frac{1}{2}]$. We denote by $P_{\leq1/2}$ the low frequency projection, i.e., $\widehat{P_{\leq1/2}f}(\xi)=\varphi(\xi)\hat{f}(\xi)$, and let $P_{>1/2}=1-P_{\leq1/2}$ be the high frequency projection. Then, by the standard Gagliardo-Nirenberg inequalities 
\begin{equation}\label{standard GN}
\|u\|_{L^4}^4\lesssim\|u\|_{L^2}^{\frac{4\sigma-1}{\sigma}}\||\nabla|^\sigma u\|_{L^2}^{\frac{1}{\sigma}},\ \|u\|_{L^4}^4\lesssim\|u\|_{L^2}^3\||\nabla| u\|_{L^2}
\end{equation}
and by Lemma \ref{p_k bound}, we prove that
\begin{align*}
\|u\|_{L^4}^4&\lesssim \|P_{>1/2}u\|_{L^4}^4+\|P_{\leq1/2}u\|_{L^4}^4\\
&\lesssim \|P_{>1/2}u\|_{L^2}^{\frac{4\sigma-1}{\sigma}}\||\nabla|^\sigma P_{>1/2}u\|_{L^2}^{\frac{1}{\sigma}}+\|P_{\leq1/2}u\|_{L^2}^3\|\nabla P_{\leq1/2}u\|_{L^2}\\
&\lesssim \|u\|_{L^2}^{\frac{4\sigma-1}{\sigma}}\|\mathcal{P}_1^{1/2} u\|_{L^2}^{\frac{1}{\sigma}}+\|u\|_{L^2}^3\|\mathcal{P}_1^{1/2} u\|_{L^2},
\end{align*}
where the implicit constants may depend on $\sigma$.
\end{proof}

Next, we introduce the functionals to be minimized. Let $\theta\in(0,1)$ be a number sufficiently close to $1$. We define $\mathcal{W}(u)$ by
\begin{equation}\label{mathcal W}
\mathcal{W}(u):=W_1(u)+W_2(u)-W_3(u),
\end{equation}
where
$$W_1(u):=\frac{\|u\|_{L^2}^{\frac{4\sigma-1}{\sigma}}\|\mathcal{P}_1^{1/2} u\|_{L^2}^{\frac{1}{\sigma}}}{\|u\|_{L^4}^4},\ W_2(u):=\alpha\frac{\|u\|_{L^2}^3\|\mathcal{P}_1^{1/2} u\|_{L^2}}{\|u\|_{L^4}^4}$$
and
$$W_3(u)=W_1(u)^{1-\theta}W_2(u)^\theta.$$
\begin{remark}
In view of the previous Gagliardo-Nirenberg inequality, it is natural to consider the functional $\mathcal{W}$. The reason why we introduced the interpolation term $W_3$ comes from the fact that we will construct the moving soliton using a profile decomposition and to reach a contradiction we need to somehow use minimization properties of well-known ground states associated to our problem. This third term $W_3$ will help in getting the contradiction. Notice that since $\theta$ will be taken close to $1$, the third term $W_3$ is close to $W_2$. 
\end{remark}

We will find a moving soliton from the minimization problem
$$\frac{1}{c_{GN}}=\inf_{u\in H^1}\mathcal{W}(u).$$
Note that $c_{GN}$ gives the sharp constant for Gagliardo-Nirenberg inequality (Theorem \ref{GN for P_1}).

We now show that any minimizer of the minimization problem is a weak solution of the desired Euler-Lagrange equation. 
\begin{lemma}\label{EL equation}
Suppose that $u$ is a minimizer of $\mathcal W(u)$ in $H^\sigma$. Then, $u$ is a weak solution to
$$\mathcal{P}_1 u +a u-bu^3=0$$
for some  $a, b>0$. Thus, $Q_1=b^{1/2} u$ solves
$$\mathcal{P}_1 Q_1+\omega^{2\sigma}Q_1-Q_1^3=0,\quad \omega=a^{1/2\sigma}.$$
\end{lemma}
\begin{proof}
If $u$ is a minimizer, then for any compactly supported function $v$ on $\mathbb R$, we have
\begin{align*}
&\frac{d}{d\epsilon}\Big|_{\epsilon=0}W_1(u+\epsilon v)\\
&=\frac{d}{d\epsilon}\Big|_{\epsilon=0}\frac{\big(\|u+\epsilon v\|_{L^2}^2\big)^{\frac{4\sigma-1}{2\sigma}}\big(\|\mathcal{P}_1^{1/2}(u+\epsilon v)\|_{L^2}^2\big)^{\frac{1}{2\sigma}}}{\|u+\epsilon v\|_{L^4}^4}\\
&=\frac{(4\sigma-1)W_1(u)}{2\sigma\|u\|_{L^2}^2}\frac{d}{d\epsilon}\Big|_{\epsilon=0}\|u+\epsilon v\|_{L^2}^2+\frac{W_1(u)}{2\sigma\|\mathcal{P}_1^{1/2}u\|_{L^2}^2}\frac{d}{d\epsilon}\Big|_{\epsilon=0}\|\mathcal{P}_1^{1/2}(u+\epsilon v)\|_{L^2}^2\\
&\quad-\frac{W_1(u)}{\|u\|_{L^4}^4}\frac{d}{d\epsilon}\Big|_{\epsilon=0}\|u+\epsilon v\|_{L^4}^4\\
&=\frac{(4\sigma-1)W_1(u)}{\sigma \|u\|_{L^2}^2}\Re\int_{\mathbb{R}} u\bar{v}dx+\frac{W_1(u)}{\sigma\|\mathcal{P}_1^{1/2}u\|_{L^2}^2}\Re\int_{\mathbb{R}}(\mathcal{P}_1 u)\bar{v} dx-\frac{W_1(u)}{\|u\|_{L^4}^4}\Re\int_{\mathbb{R}}|u|^2u\bar{v}dx.
\end{align*}
Similarly, we compute $\frac{d}{d\epsilon}\big|_{\epsilon=0}W_2(u+\epsilon v)$ and $\frac{d}{d\epsilon}\big|_{\epsilon=0}W_3(u+\epsilon v)$. Then, collecting all, we obtain 
\begin{equation}
\begin{aligned}
0&=\frac{d}{d\epsilon}\Big|_{\epsilon=0}\mathcal{W}(u+\epsilon v)\\
&=\frac{1}{\|u\|_{L^2}^2}\Big\{\tfrac{4\sigma-1}{\sigma}W_1(u)+3W_2(u)-\big(\tfrac{4\sigma-1}{\sigma}(1-\theta)+3\theta\big)W_3(u)\Big\}\Re\int_{\mathbb{R}} u \overline{v} dx\\
&\quad+\frac{1}{\|\mathcal{P}_1^{1/2}u\|_{L^2}^2}\Big\{\tfrac{1}{\sigma}W_1(u)+W_2(u)-\big(\tfrac{1}{\sigma}(1-\theta)+\theta\big)W_3(u)\Big\}\Re\int_{\mathbb{R}}(\mathcal{P}_1 u) \overline{v}dx\\
&\quad-\frac{1}{\|u\|_{L^4}^4}\Big\{W_1(u)+W_2(u)-W_3(u)\Big\}\Re\int_{\mathbb{R}} u^3 \overline{v}dx.
\end{aligned}
\end{equation}
By Young's inequality, one can show that all the terms in front of the integrals are nonnegative, provided that $\theta<1$ is close enough to 1. Hence, $u$ solves $\Re\{\mathcal{P}_1 u +a u-bu^3\}=0$. Similarly, computing $\frac{d}{d\epsilon}|_{\epsilon=0}\mathcal{W}(u+i\epsilon v)=0$, we prove that $\Im\{\mathcal{P}_1 u +a u-bu^3\}=0$.
\end{proof}

\section{Existence of a Static Ground State}

In this section, in order to illustrate our strategy to prove the main theorem, we give a proof of existence of a static ground state, which is a well-known fact in the literature (see for instance \cite{FL} and the references therein).

\begin{theorem}[Existence of a static ground state]\label{existence of a static ground state}
For $\sigma\in(\frac{1}{4},1]$, there exists a minimizer $Q\in H^\sigma$ for the Weinstein functional
$$W(u):=\frac{\|u\|_{L^2}^{\frac{4\sigma-1}{\sigma}}\||\nabla|^\sigma u\|_{L^2}^{\frac{1}{\sigma}}}{\|u\|_{L^4}^4}$$
such that $Q$ solves the ground state equation
\begin{equation}\label{static elliptic equation}
(-\Delta)^\sigma Q+Q-Q^3=0.
\end{equation}
\end{theorem}

We will prove the theorem by the standard concentration-compactness argument, developed by P.L. Lions \cite{lions1,lions2}, in the form of the bubble decomposition as in \cite{gerard} (see also \cite{HK, KV}).

\begin{proposition}[Bubble decomposition]\label{bubble decomposition}
Let $\sigma\in(\frac{1}{4},1]$. Suppose that $\{u_n\}_{n=1}^\infty$ be a bounded sequence in $H^\sigma$. Then there exist a subsequence of $\{u_n\}_{n=1}^\infty$ (still denoted by $\{u_n\}_{n=1}^\infty$), a family of sequences of translation parameters $\{x_n^j\}_{n=1}^\infty$, $j=1,2,\cdots$, and a sequence of $H^\sigma$ functions $\{\phi^j\}_{j=1}^\infty$ such that for every $J\geq 1$,
$$u_n(x)=\sum_{j=1}^J \phi^j(x-x_n^j)+R_n^J$$
satisfying the following properties:\\
$(i)$ (Asymptotic orthogonality) For each $j\neq j'$, $|x_n^j-x_n^{j'}|\to\infty$ as $n\to\infty$. Moreover, for $0\leq j\leq J$, $R_n^J(\cdot+x_n^j)\rightharpoonup 0$ in $H^\sigma$ as $n\to\infty$.\\
$(ii)$ (Remainder estimate)
$$\lim_{J\to\infty}\limsup_{n\to\infty}\|R_n^J\|_{L^4}=0.$$
$(iii)$ (Asymptotic Pythagorean expansions)
\begin{align*}
\|u_n\|_{L^2}^2&=\sum_{j=1}^J\|\phi^j\|_{L^2}^2+\|R_n^J\|_{L^2}^2+o_n(1),\\
\||\nabla|^\sigma u_n\|_{L^2}^2&=\sum_{j=1}^J\||\nabla|^\sigma\phi^j\|_{L^2}^2+\||\nabla|^\sigma R_n^J\|_{L^2}^2+o_n(1),\\
\|u_n\|_{L^4}^4&=\sum_{j=1}^J\|\phi^j\|_{L^4}^4+\|R_n^J\|_{L^4}^4+o_n(1).
\end{align*}
\end{proposition}

\begin{proof}[Proof of Theorem \ref{existence of a static ground state}]
\textit{Step 1. (Formulation of a two-bubble decomposition)} Let $\{u_n\}_{n=1}^\infty$ be a minimizing sequence in $H^\sigma$ for $W(u)$, i.e., $W(u_n)\to\frac{1}{c_{GN}}$, where $c_{GN}$ is the sharp constant for the Gagliardo-Nirenberg inequality
\begin{equation}\label{GN inequality}
\|u\|_{L^4}^4\leq c\|u\|_{L^2}^{\frac{4\sigma-1}{\sigma}}\||\nabla|^\sigma u\|_{L^2}^{\frac{1}{\sigma}}.
\end{equation}
Since the functional $W(u)$ is invariant under multiplication by a constant and scaling, i.e., $W(\alpha u(\frac{\cdot}{\lambda}))=W(u)$ for all $\alpha\in\mathbb{C}$ and $\lambda>0$, replacing $u_n$ by $\alpha_nu_n(\lambda_n\cdot)$ with $\alpha_n=\frac{\lambda_n^{1/2}}{\|u_n\|_{L^2}}$ and $\lambda_n=\big(\frac{\|u_n\|_{L^2}}{\||\nabla|^{\sigma}u_n\|_{L^2}}\big)^{1/\sigma}$, we may assume that $\|u_n\|_{L^2}=\||\nabla|^\sigma u_n\|_{L^2}=1$ for each $n$. Obviously, $\{u_n\}_{n=1}^\infty$ is bounded in $H^\sigma$.

Now, applying the bubble decomposition (Proposition \ref{bubble decomposition}) with $J=1$ to the bounded sequence $\{u_n\}_{n=1}^\infty$, extracting a subsequence, we write
$$u_n=\phi(\cdot-x_n)+R_n,\ \phi\neq 0.$$
Indeed, if there is no such non-zero $\phi$ $(\Rightarrow u_n=R_n)$, then by Proposition \ref{bubble decomposition} $(ii)$,
$$W(u_n)=\frac{\|u_n\|_{L^2}^{\frac{4\sigma-1}{\sigma}}\||\nabla|^\sigma u_n\|_{L^2}^{\frac{1}{\sigma}}}{\|u_n\|_{L^4}^4}=\frac{1}{\|R_n\|_{L^4}^4}\to\infty,$$
which contradicts the minimality of $\{u_n\}_{n=1}^\infty$. Moreover, we may replace $u_n$ by $u_n(\cdot+x_n)$ because of translation invariance of the functional $W$, i.e., $W((u(\cdot-b))=W(u)$ for $b\in\mathbb{R}$. As a result, we get a two-bubble decomposition
$$u_n=\phi+R_n$$
with the properties in Proposition \ref{bubble decomposition}. Here, extracting a subsequence if necessary, we may assume that $\|R_n\|_{L^2}$, $\||\nabla|^\sigma R_n\|_{L^2}$ and $\|R_n\|_{L^4}^4$ are convergent as $n\to\infty$.

Suppose that $R_n\to 0$ in $H^\sigma$. Then, $\phi$ is a minimizer for $W(u)$, because if $u_n\to \phi$ in $H^\sigma$ (so, $u_n\to \phi$ in $L^4$ by the Sobolev imbedding), then
$$\frac{1}{c_{GN}}=\lim_{n\to\infty}W(u_n)=\frac{\|u_n\|_{L^2}^{\frac{4\sigma-1}{\sigma}}\||\nabla|^\sigma u_n\|_{L^2}^{\frac{1}{\sigma}}}{\|u_n\|_{L^4}^4}=\frac{\|\phi\|_{L^2}^{\frac{4\sigma-1}{\sigma}}\||\nabla|^\sigma \phi\|_{L^2}^{\frac{1}{\sigma}}}{\|\phi\|_{L^4}^4}=W(\phi).$$
Thus, we only need to preclude the dichotomy scenario that $\{u_n\}_{n=1}^\infty$ has two non-negligible profiles, in other words,
\begin{equation}\label{static contradiction assumption}
\lim_{n\to\infty}\|R_n\|_{L^2},\lim_{n\to\infty}\||\nabla|^\sigma R_n\|_{L^2}\geq \epsilon>0.
\end{equation}
\textit{Step 2. (Elimination of the dichotomy scenario)} Suppose that \eqref{static contradiction assumption} holds. Then, we aim to deduce a contradiction by showing that $\frac{d}{dt}\big|_{t=1}W(t\phi+R_n)>c>0$. Indeed, if so, there exists small $\delta>0$ such that $W((1-\delta)\phi+R_n)\leq W(u_n)-\frac{c}{2}\delta$, which contradicts the minimality of $\{u_n\}_{n=1}^\infty$.

To this end, we compute
\begin{equation}\label{static derivative calc 1}
\begin{aligned}
\frac{d}{dt}\Big|_{t=1}\|t\phi+R_n\|_{L^2}^2&=\frac{d}{dt}\Big|_{t=1} \Big\{t^2 \|\phi\|_{L^2}^2+2t\Re \int_{\mathbb{R}} \phi\overline{R_n} dx+\|R_n\|_{L^2}^2\Big\}\\
&=2 \|\phi\|_{L^2}^2+2\Re \int_{\mathbb{R}}\phi\overline{R_n} dx\\
&=2\|\phi\|_{L^2}^2+o_n(1)\quad\textup{(by Proposition \ref{bubble decomposition} $(ii)$)}
\end{aligned}
\end{equation}
and similarly,
\begin{equation}\label{static derivative calc 2}
\frac{d}{dt}\Big|_{t=1}\||\nabla|^\sigma(t\phi+R_n)\|_{L^2}^2=\cdots=2\||\nabla|^\sigma\phi\|_{L^2}^2+o_n(1).
\end{equation}
Moreover, we have
\begin{equation}\label{static derivative calc 3}
\frac{d}{dt}\Big|_{t=1}\|t\phi+R_n\|_{L^4}^4=4\|\phi\|_{L^4}^4+\textup{cross terms}.
\end{equation}
We claim that each cross term in \eqref{static derivative calc 3} converges to zero as $n\to\infty$, and thus
\begin{equation}\label{static derivative calc 3'}
\frac{d}{dt}\Big|_{t=1}\|t\phi+R_n\|_{L^4}^4=4\|\phi\|_{L^4}^4+o_n(1).
\end{equation}
Among cross terms in \eqref{static derivative calc 3}, we only consider
$$\int_{\mathbb{R}} \phi|R_n|^2 \overline{R_n}dx,$$
since other cross terms can be treated similarly. By the profile decomposition, passing to a subsequence,
$$u_n(x)=\phi+R_n=\phi+\sum_{j=2}^J \phi^j(x-x_n^j)+R_n^J(x),$$
where $|x_n^j|\to \infty$ (since the first profile $\phi$ is not translated) and $|x_n^{j}-x_n^{j'}|\to\infty$ for $j\neq j'$. Thus, by the asymptotic orthogonality of $\{x_n^j\}_{n=1}^\infty$ and H\"older inequality, 
\begin{align*}
\Big|\int_{\mathbb{R}}\phi|R_n|^2 \overline{R_n}dx\Big|&\leq\sum_{j_1,j_2,j_3=2}^J\Big|\int_{\mathbb{R}} \phi(x) \overline{\phi^{j_1}(\cdot-x_n^{j_1})}\phi^{j_2}(x-x_n^{j_2})\overline{\phi^{j_3}(\cdot-x_n^{j_3})}dx\Big|\\
&\quad+\Big|\int_{\mathbb{R}}\phi|R_n^J|^2\overline{R_n^J}dx\Big|\\
&\leq o_n(1)+\|\phi\|_{L^4}\|R_n^J\|_{L^4}^3\to 0\textup{ as }n, J\to\infty.
\end{align*}

Using \eqref{static derivative calc 1}, \eqref{static derivative calc 2} and \eqref{static derivative calc 3'}, we compute
\begin{equation}\label{static derivative calc}
\begin{aligned}
&\frac{d}{dt}\Big|_{t=1}W(t\phi+R_n)\\
&=\frac{d}{dt}\Big|_{t=1}\frac{\|t\phi+R_n\|_{L^2}^{\frac{4\sigma-1}{\sigma}}\||\nabla|^\sigma (t\phi+R_n)\|_{L^2}^{\frac{1}{\sigma}}}{\|t\phi+R_n\|_{L^4}^4}\\
&=\frac{d}{dt}\Big|_{t=1}\frac{\Big\{t^2\|\phi\|_{L^2}^2+\|R_n\|_{L^2}^2\Big\}^{\frac{4\sigma-1}{2\sigma}}\Big\{t^2\||\nabla|^\sigma \phi\|_{L^2}^2+\||\nabla|^\sigma R_n\|_{L^2}^2\Big\}^{\frac{1}{2\sigma}}}{t^4\|\phi\|_{L^4}^4+\|R_n\|_{L^4}^4}+o_n(1)\\
&=\cdots\\
&=W(u_n)\Big\{\frac{4\sigma-1}{\sigma}\frac{\|\phi\|_{L^2}^2}{\|u_n\|_{L^2}^2}+\frac{1}{\sigma}\frac{\||\nabla|^\sigma \phi\|_{L^2}^2}{\||\nabla|^\sigma u_n\|_{L^2}^2}-4\frac{\|\phi\|_{L^4}^4}{\|u_n\|_{L^4}^4}\Big\}+o_n(1).
\end{aligned}
\end{equation}
Then, by the Gagliardo-Nirenberg inequality \eqref{GN inequality}, the choice of $\{u_n\}_{n=1}^\infty$ ($W(u_n)\to \frac{1}{c_{GN}}$) and Young's inequality $ab\leq\frac{4\sigma-1}{4\sigma}a^{\frac{4\sigma}{4\sigma-1}}+\frac{1}{4\sigma}b^{4\sigma}$, we have
\begin{align*}
\frac{\|\phi\|_{L^4}^4}{\|u_n\|_{L^4}^4}&\leq\frac{c_{GN}\|\phi\|_{L^2}^{\frac{4\sigma-1}{\sigma}}\||\nabla|^\sigma \phi\|_{L^2}^{\frac{1}{\sigma}}}{c_{GN}\|u_n\|_{L^2}^{\frac{4\sigma-1}{\sigma}}\||\nabla|^\sigma u_n\|_{L^2}^{\frac{1}{\sigma}}}+o_n(1)\\
&=\Big(\frac{\|\phi\|_{L^2}}{\|u_n\|_{L^2}}\Big)^{\frac{4\sigma-1}{\sigma}}\Big(\frac{\||\nabla|^\sigma \phi\|_{L^2}}{\||\nabla|^\sigma u_n\|_{L^2}}\Big)^{\frac{1}{\sigma}}+o_n(1)\\
&\leq \frac{4\sigma-1}{4\sigma}\Big(\frac{\|\phi\|_{L^2}}{\|u_n\|_{L^2}}\Big)^4+\frac{1}{4\sigma}\Big(\frac{\||\nabla|^\sigma \phi\|_{L^2}}{\||\nabla|^\sigma u_n\|_{L^2}}\Big)^4+o_n(1).
\end{align*}
Inserting this inequality into \eqref{static derivative calc}, we obtain
\begin{align*}
&\frac{d}{dt}\Big|_{t=1}W(t\phi+R_n)\\
&\geq \frac{1}{c_{GN}}\Big\{\frac{4\sigma-1}{\sigma}\frac{\|\phi\|_{L^2}^2}{\|u_n\|_{L^2}^2}\Big(1-\frac{\|\phi\|_{L^2}^2}{\|u_n\|_{L^2}^2}\Big)+\frac{1}{\sigma}\frac{\||\nabla|^\sigma \phi\|_{L^2}^2}{\||\nabla|^\sigma u_n\|_{L^2}^2}\Big(1-\frac{\||\nabla|^\sigma \phi\|_{L^2}^2}{\||\nabla|^\sigma u_n\|_{L^2}^2}\Big)\Big\}+o_n(1)\\
&=\frac{1}{c_{GN}}\Big\{\frac{4\sigma-1}{\sigma}\|\phi\|_{L^2}^2\|R_n\|_{L^2}^2+\frac{1}{\sigma}\||\nabla|^\sigma \phi\|_{L^2}^2\||\nabla|^\sigma R_n\|_{L^2}^2\Big\}+o_n(1),
\end{align*}
where in the last identity, we used Proposition \ref{bubble decomposition} $(iv)$ and $\|u_n\|_{L^2}=\||\nabla|^\sigma u_n\|_{L^2}=1$. Therefore, by the assumption \eqref{static contradiction assumption}, we conclude that
$$\frac{d}{dt}\Big|_{t=1}W(t\phi+R_n)\geq \frac{4\sigma-1}{c_{GN}\sigma}\|\phi\|_{L^2}^2\epsilon+o_n(1)\textup{, or }\geq\frac{1}{c_{GN}\sigma}\||\nabla|^\sigma \phi\|_{L^2}^2\epsilon+o_n(1),$$
which contradicts from minimality of $\{u_n\}_{n=1}^\infty$.\\
\textit{Step 3. (Euler-Lagrange equation)} From Step 1 and 2, we obtained a minimizer for the Weinstein functional $W(u)$. Hence, modifying the minimizer $\phi$ by $Q=\alpha\phi(\lambda\cdot)$ as in the proof of Lemma \ref{EL equation}, one can show that there is a minimizer solving the equation \eqref{static elliptic equation}.
\end{proof}

\section{Proof of Theorem \ref{existence}}

We prove our main theorem following the argument in the proof of Theorem \ref{existence of a static ground state}.

\subsection{Two-bubble decomposition}
The first step is to formulate a two-bubble decomposition for a minimizing sequence. The main difference from the two-bubble decomposition in the previous section is that it has a scaling parameter $\lambda_n$ (see \eqref{two-bubble decomposition}). The reason we need $\lambda_n$ is that unlike the previous case, one cannot modify a minimizing sequence easily by scaling to make it bounded in $H^\sigma$, because the Weinstein functional $\mathcal{W}(u)$ is not scaling-invariant.

\begin{lemma}[Two-bubble decomposition]
There exist a minimizing sequence $\{u_n\}_{n=1}^\infty$ for the Weinstein functional $\mathcal{W}(u)$ and a sequence of scaling parameters $\{\lambda_n\}_{n=1}^\infty$, with $\lambda_n=1$, $\lambda_n\to 0$ or $\lambda_n\to+\infty$, such that 
\begin{equation}\label{two-bubble decomposition}
v_n(x):=u_n(\lambda_nx)=\phi(x)+R_n(x),\quad\phi\neq0,
\end{equation}
satisfying the asymptotic Pythagorean expansions,
\begin{equation}\label{Properties of two profiles}
\begin{aligned}
\|v_n\|_{L^2}^2&=\|\phi\|_{L^2}^2+\|R_n\|_{L^2}^2+o_n(1),\\
\||\nabla|^\sigma v_n\|_{L^2}^2&=\||\nabla|^\sigma\phi\|_{L^2}^2+\|R_n\|_{L^2}^2+o_n(1),\\
\|\mathcal{P}_{\lambda_n}^{1/2} v_n\|_{L^2}^2&=\|\mathcal{P}_{\lambda_n}^{1/2}\phi\|_{L^2}^2+\|\mathcal{P}_{\lambda_n}^{1/2} R_n\|_{L^2}^2+o_n(1),\\
\|v_n\|_{L^4}^4&=\|\phi\|_{L^4}^4+\|R_n\|_{L^4}^4+o_n(1).
\end{aligned}
\end{equation}
\end{lemma}

\begin{proof}
Let $\{u_n\}_{n=1}^\infty\subset H^\sigma$ be a minimizing sequence for the Weinstein functional $\mathcal{W}(u)$. Then, as we did in the proof of Theorem \ref{existence of a static ground state}, for each $n$, we can find $\alpha_n,\mu_n>0$ such that
$$\|\alpha_n u_n(\mu_n\cdot)\|_{L^2}=\||\nabla|^\sigma(\alpha_n u_n(\mu_n\cdot))\|_{L^2}=1.$$
Since $\mathcal{W}(\alpha u)=\mathcal{W}(u)$, replacing $u_n$ by $\alpha_n u_n$, we may assume that $\alpha_n=1$. Moreover, passing to a subsequence, we may assume that $\mu_n\to \mu_*\in[0,+\infty)$ or $\mu_n\to+\infty$. If $\mu_*\in(0,+\infty)$, we set $v_n=u_n$ and $\lambda_n=1$. Otherwise, we set $v_n=u_n(\lambda_n\cdot)$ with $\lambda_n=\mu_n$. Then, $v_n=u_n(\lambda_n\cdot)$ satisfies
\begin{equation}\label{normalization}
\|v_n\|_{L^2}\sim\||\nabla|^\sigma v_n\|_{L^2}\sim1,\textup{ and }\lambda_n=1,\ \lambda_n\to 0\textup{ or }\lambda_n\to+\infty.
\end{equation}
Thus, by Proposition \ref{bubble decomposition}, passing to a subsequence, we can write $v_n=\phi(\cdot-x_n)+R_n$. Moreover, by the translation invariance of $\mathcal{W}(u)$, i.e., $\mathcal{W}(u(\cdot-a))=\mathcal{W}(u)$, replacing $u_n$  by $u_n(\cdot+x_n)$, we obtain a simpler two-bubble decomposition
$$u_n(\lambda_n\cdot)=v_n=\phi+R_n$$
satisfying \eqref{Properties of two profiles} except the third identity.

To show that there is a two-bubble decomposition with $\phi\neq 0$, we assume that there is no such $\phi$ for contradiction, in other word, $v_n=R_n$. Then, by Proposition \ref{bubble decomposition} $(ii)$ with $R_n=R_n^J$, we have $\|v_n\|_{L^4}\to 0$ up to a subsequence. When $\lambda_n=1$ or $\lambda_n\to0$, we deduce a contradiction to the minimality of $\{u_n\}_{n=1}^\infty$ by showing $W_1(u_n)\to\infty$. Indeed, by scaling and \eqref{normalization}, we have
$$W_1(u_n)=\frac{\|v_n\|_{L^2}^{\frac{4\sigma-1}{\sigma}}\|(\mathcal{P}_{\lambda_n})^{1/2}v_n\|_{L^2}^{\frac{1}{\sigma}}}{\|v_n\|_{L^4}^4}\sim\frac{\|(\mathcal{P}_{\lambda_n})^{1/2}v_n\|_{L^2}^{\frac{1}{\sigma}}}{\|v_n\|_{L^4}^4}.$$
Suppose that $\lambda_n=1$. Then, by modifying the proof of Lemma \ref{p_k bound} $(ii)$, one can show that $p_1(\xi)\sim_c|\xi|^{2\sigma}$ for $|\xi|\geq c$. Thus, for sufficiently large $n$,
\begin{align*}
\|(\mathcal{P}_1)^{1/2}v_n\|_{L^2}^2&\geq\int_{|\xi|\geq c} p_1(\xi)|\hat{v}_n(\xi)|^2 d\xi\\
&\sim_c\int_{|\xi|\geq c}|\xi|^{2\sigma}|\hat{v}_n(\xi)|^2 d\xi=\||\nabla|^\sigma v_n\|_{L^2}^2-\int_{|\xi|\leq c}|\xi|^{2\sigma}|\hat{v}_n(\xi)|^2 d\xi\\
&\geq \||\nabla|^\sigma v_n\|_{L^2}^2-c^{2\sigma}\|v_n\|_{L^2}^2\geq\delta>0,
\end{align*}
provided that $c>0$ is small enough. Similarly, if $\lambda_n\to0$, by Lemma \ref{p_k bound} $(ii)$, 
\begin{align*}
\|(\mathcal{P}_{\lambda_n})^{1/2}v_n\|_{L^2}&\geq\int_{|\xi|\geq c\lambda_n } p_{\lambda_n}(\xi)|\hat{v}_n(\xi)|^2 d\xi\sim_c\int_{|\xi|\geq c\lambda_n}|\xi|^{2\sigma}|\hat{v}_n(\xi)|^2 d\xi\\
&\geq\int_{|\xi|\geq c}|\xi|^{2\sigma}|\hat{v}_n(\xi)|^2 d\xi=\||\nabla|^\sigma v_n\|_{L^2}^2-\int_{|\xi|\leq c}|\xi|^{2\sigma}|\hat{v}_n(\xi)|^2 d\xi\\
&\geq \||\nabla|^\sigma v_n\|_{L^2}^2-c^{2\sigma}\|v_n\|_{L^2}^2\geq\delta>0.
\end{align*}
In both cases, we obtain
$$W_1(u_n)\gtrsim\frac{1}{\|v_n\|_{L^4}^4}\to \infty.$$
Suppose that $\lambda_n\to\infty$. Then by scaling and \eqref{normalization}, we write
$$W_2(u_n)=\frac{\lambda_n^{1-\sigma}\|v_n\|_{L^2}^3\|(\mathcal{P}_{\lambda_n})^{1/2}v_n\|_{L^2}}{\|v_n\|_{L^4}^4}\sim\frac{\lambda_n^{1-\sigma}\|(\mathcal{P}_{\lambda_n})^{1/2}v_n\|_{L^2}}{\|v_n\|_{L^4}^4}.$$
We claim that if $|\xi|\geq c$, then
$$p_{\lambda_n}(\xi)\gtrsim_c\lambda_n^{2\sigma-2}|\xi|^{2\sigma}.$$
Indeed, if $|\xi|\geq\frac{\lambda_n}{2}$, then it follows from Lemma \ref{p_k bound} $(ii)$ that $|p_{\lambda_n}(\xi)|\sim|\xi|^{2\sigma}$. If $c\leq|\xi|\leq\frac{\lambda_n}{2}$, then using the second order Taylor expansion, one can show that $p_{\lambda_n}(\xi)=\lambda_n^{2\sigma}(|\frac{\xi}{\lambda_n}+1|^{2\sigma}-1-2\sigma\frac{\xi}{\lambda_n})\sim \lambda_n^{2\sigma}(\frac{\xi}{\lambda_n})^2=\lambda_n^{2\sigma-2}|\xi|^2\gtrsim_c\lambda_n^{2\sigma-2}|\xi|^{2\sigma}$. Therefore, by the claim, we prove that
\begin{align*}
W_2(u_n)&\gtrsim_c\frac{1}{\|v_n\|_{L^4}^4}\Big\{\int_{|\xi|\geq c}|\xi|^{2\sigma}|\hat{v}_n(\xi)|^2d\xi\Big\}^{1/2}\\
&=\frac{1}{\|v_n\|_{L^4}^4}\Big\{\||\nabla|^\sigma v_n\|_{L^2}^2-\int_{|\xi|\leq c}|\xi|^{2\sigma}|\hat{v}_n(\xi)|^2d\xi\Big\}^{1/2}\\
&\geq\frac{\big\{\||\nabla|^\sigma v_n\|_{L^2}^2-c^{2\sigma}\|v_n\|_{L^2}^2\big\}^{1/2}}{\|v_n\|_{L^4}^4}\geq\frac{\delta}{\|v_n\|_{L^4}^4}\to\infty.
\end{align*}

For the third identity in \eqref{Properties of two profiles}, expanding its left hand side, we see that it suffices to show that the cross term $\la\mathcal{P}_{\lambda_n} \phi, R_n\ra_{L^2}$ converges to zero. It is obvious when $\lambda_n=1$. Suppose that $\lambda_n\to0$. Then, since $\|R_n\|_{H^\sigma}$ is bounded and $R_n\rightharpoonup 0$ in $H^\sigma$, it is enough to show that $\la\nabla\ra^{-\sigma}\mathcal{P}_{\lambda_n}\phi$ converges in $L^2$. Indeed, by Lemma \ref{p_k bound} $(iii)$, we have a point-wise bound $|\la\xi\ra^{-\sigma}p_{\lambda_n}(\xi)|\lesssim\la\xi\ra^\sigma$, where the implicit constant is independent on $n$. Therefore, it follows from the dominated convergence theorem that $\la\nabla\ra^{-\sigma}\mathcal{P}_{\lambda_n}\phi\to\phi$ in $L^2$. Suppose that $\lambda_n\to \infty$. Then, given $\epsilon>0$, there exists $\phi_\epsilon\in H^\sigma$ such that $\|\phi-\phi_\epsilon\|_{H^\sigma}\leq\epsilon$ and $\hat{\phi}_\epsilon$ is supported in $\{c\leq |\xi|\leq C\}$. Thus, by Lemma \ref{p_k bound} $(iii)$, we have
$$\la\mathcal{P}_{\lambda_n}(\phi-\phi_\epsilon), R_n\ra_{L^2}\leq\|\la\nabla\ra^{-\sigma}(\mathcal{P}_{\lambda_n})^{1/2}(\phi-\phi_\epsilon)\|_{L^2}\|R_n\|_{H^\sigma}\lesssim \|\phi-\phi_\epsilon\|_{H^\sigma}\leq \epsilon.$$
On the other hand, we have
\begin{align*}
\la\mathcal{P}_{\lambda_n} \phi_\epsilon, R_n\ra_{L^2}&=\la p_{\lambda_n}\hat{\phi}_\epsilon, \hat{R}_n\ra_{L^2}=\lambda_n^{2\sigma}\la (|1+\tfrac{\xi}{\lambda_n}|^{2\sigma}-1-2\sigma\tfrac{\xi}{\lambda_n})\hat{\phi}_\epsilon, \hat{R}_n\ra_{L^2}\\
&=\lambda_n^{2\sigma-2}\sigma(2\sigma-1)\la-\Delta \phi_\epsilon, R_n\ra_{L^2}\\
&\quad+\lambda_n^{2\sigma-2}\big\la\big(|1+\tfrac{\xi}{\lambda_n}|^{2\sigma}-1-2\sigma\tfrac{\xi}{\lambda_n}-\sigma(2\sigma-1)\tfrac{\xi^2}{\lambda_n^2}\big)\hat{\phi}_\epsilon, \hat{R}_n\big\ra_{L^2}\\
&=I_n+II_n.
\end{align*}
It is easy to see that $I_n$ converges to zero, since  $\la-\Delta \phi_\epsilon, R_n\ra_{L^2}\leq \|\Delta\phi_\epsilon\|_{L^2}\|R_n\|_{L^2}\lesssim_\epsilon1$ and $\lambda_n^{2\sigma-2}\to 0$. Moreover, using the dominated convergence theorem as above, one can show that $II_n$ converges to zero. Since $\epsilon>0$ is arbitrary, we conclude that $\la\mathcal{P}_{\lambda_n} \phi, R_n\ra_{L^2}\to 0$.
\end{proof}

\subsection{Concentration}
For contradiction, we assume that $u_n=v_n(\frac{\cdot}{\lambda_n})$ with $\lambda_n\to0$. Then, we prove that 
\begin{equation}\label{lambda to zero}
\mathcal{W}(u_n)\to\frac{\|Q_{(\sigma)}\|_{L^2}^{\frac{4\sigma-1}{\sigma}}\||\nabla|^\sigma Q_{(\sigma)}\|_{L^2}^{\frac{1}{\sigma}}}{\|Q_{(\sigma)}\|_{L^4}^4},
\end{equation}
where $Q_{(\sigma)}$ is the ground state for the elliptic equation
$$(-\Delta)^\sigma Q_{(\sigma)}+Q_{(\sigma)}-Q_{(\sigma)}^3=0.$$
Indeed, by scaling,
\begin{equation}\label{scaled W}
\begin{aligned}
\mathcal{W}_1(u_n)&=\frac{\|v_n\|_{L^2}^{\frac{4\sigma-1}{\sigma}}\|(\mathcal{P}_{\lambda_n})^{1/2}v_n\|_{L^2}^{\frac{1}{\sigma}}}{\|v_n\|_{L^4}^4}+\alpha\lambda_n^{2-2\sigma}\frac{\|v_n\|_{L^2}^{\frac{4\sigma-1}{\sigma}}\|(\mathcal{P}_{\lambda_n})^{1/2}v_n\|_{L^2}^{\frac{1}{\sigma}}}{\|v_n\|_{L^4}^4}\\
&-\Big(\frac{\|v_n\|_{L^2}^{\frac{4\sigma-1}{\sigma}}\|(\mathcal{P}_{\lambda_n})^{1/2}v_n\|_{L^2}^{\frac{1}{\sigma}}}{\|v_n\|_{L^4}^4}\Big)^{1-\theta}\Big(\alpha\lambda_n^{2-2\sigma}\frac{\|v_n\|_{L^2}^{\frac{4\sigma-1}{\sigma}}\|(\mathcal{P}_{\lambda_n})^{1/2}v_n\|_{L^2}^{\frac{1}{\sigma}}}{\|v_n\|_{L^4}^4}\Big)^\theta.
\end{aligned}
\end{equation}
We claim that
$$\|(\mathcal{P}_{\lambda_n})^{1/2}v_n\|_{L^2}^2=\||\nabla|^\sigma v_n\|_{L^2}^2+o_n(1).$$
To show the claim, we write 
\begin{align*}
&\|(\mathcal{P}_{\lambda_n})^{1/2}v_n\|_{L^2}^2-\||\nabla|^\sigma v_n\|_{L^2}^2\\
&=\int_{|\xi|\leq C\lambda_n}+ \int_{|\xi|\geq C\lambda_n} \big(|\xi+\lambda_n|^{2\sigma}-\lambda_n^{2\sigma}-2\sigma\lambda_n^{2\sigma-1}\xi-|\xi|^{2\sigma}\big)|\hat{v}_n(\xi)|^2d\xi\\
&=I_n+II_n.
\end{align*}
Using Taylor series, one can show that for any $\epsilon>0$, there exists $C>0$ such that 
\begin{align*}
II_n&=\int_{\frac{\lambda_n}{|\xi|}\leq \frac{1}{C}} |\xi|^{2\sigma}\big(|1+\tfrac{\lambda_n}{|\xi|}|^{2\sigma}-(\tfrac{\lambda_n}{|\xi|})^{2\sigma}-2\sigma(\tfrac{\lambda_n}{|\xi|})^{2\sigma-1}-1\big)|\hat{v}_n(\xi)|^2d\xi\\
&\lesssim\frac{1}{C^{2\sigma-1}}\||\nabla|^\sigma v_n\|_{L^2}^2\lesssim \frac{1}{C^{2\sigma-1}}\lesssim\epsilon.
\end{align*}
Moreover, given $C>0$, sending $n\to \infty$, we get
$$I_n\lesssim (C\lambda_n)^{2\sigma}\|v_n\|_{L^2}^2\sim (C\lambda_n)^{2\sigma}\to 0.$$
Going back to \eqref{scaled W}, by the claim with $\lambda_n\to 0$, we obtain
$$\mathcal{W}_1(u_n)=\frac{\|v_n\|_{L^2}^{\frac{4\sigma-1}{\sigma}}\||\nabla|^\sigma v_n\|_{L^2}^{\frac{1}{\sigma}}}{\|v_n\|_{L^4}^4}+o_n(1).$$
Then, \eqref{lambda to zero} follows from the minimization property of $Q_{(\sigma)}$, i.e.,
$$\frac{\|u\|_{L^2}^{\frac{4\sigma-1}{\sigma}}\||\nabla|^\sigma u\|_{L^2}^{\frac{1}{\sigma}}}{\|u\|_{L^4}^4}\geq\frac{\|Q_{(\sigma)}\|_{L^2}^{\frac{4\sigma-1}{\sigma}}\||\nabla|^\sigma Q_{(\sigma)}\|_{L^2}^{\frac{1}{\sigma}}}{\|Q_{(\sigma)}\|_{L^4}^4}$$
(see Theorem \ref{existence of a static ground state}).

Next, we deduce a contradiction by showing that
$$\mathcal{W}(Q_{(\sigma)})<\frac{\|Q_{(\sigma)}\|_{L^2}^{\frac{4\sigma-1}{\sigma}}\||\nabla|^\sigma Q_{(\sigma)}\|_{L^2}^{\frac{1}{\sigma}}}{\|Q_{(\sigma)}\|_{L^4}^4}.$$
To this end, we claim that there exists $c\in(0,1)$ such that
\begin{equation}\label{claim for w}
\|\mathcal{P}_1^{1/2}Q_{(\sigma)}\|_{L^2}^2= c\||\nabla|^\sigma Q_{(\sigma)}\|_{L^2}^2,
\end{equation}
where $c$ may depend on $Q_{(\sigma)}$. Indeed, by Plancherel theorem and the symmetry of $Q_{(\sigma)}$, we write
\begin{equation}\label{Plancherel for Q}
\|\mathcal{P}_1^{1/2}Q_{(\sigma)}\|_{L^2}^2=\int_{\mathbb{R}} p_1(\xi)|\hat{Q}_{(\sigma)}(\xi)|^2d\xi=\int_{\mathbb{R}}g(\xi)|\hat{Q}_{(\sigma)}(\xi)|^2d\xi,
\end{equation}
where $g(\xi):=\frac{1}{2}(p_1(\xi)+p_1(-\xi))$ is an even function. Thus, it suffices to show that $g(\xi)<|\xi|^{2\sigma}$. Indeed, if $0<\xi<1$, then using Taylor series, we write
\begin{align*}
g(\xi)&=\frac{1}{2}\Big\{|\xi+1|^{2\sigma}+|\xi-1|^{2\sigma}\Big\}-1\\
&=\Big\{1+\tfrac{2\sigma(2\sigma-1)}{2!}\xi^2+\tfrac{2\sigma(2\sigma-1)(2\sigma-2)(2\sigma-3)}{4!}\xi^4\\
&\quad\quad\quad\quad+\tfrac{2\sigma(2\sigma-1)(2\sigma-2)(2\sigma-3)(2\sigma-4)(2\sigma-5)}{6!}\xi^6+\cdots\Big\}-1\\
&=|\xi|^{2\sigma}\Big\{\tfrac{2\sigma(2\sigma-1)}{2!}\xi^{2-2\sigma}+\tfrac{2\sigma(2\sigma-1)(2\sigma-2)(2\sigma-3)}{4!}\xi^{4-2\sigma} \\
&\quad\quad\quad\quad+\tfrac{2\sigma(2\sigma-1)(2\sigma-2)(2\sigma-3)(2\sigma-4)(2\sigma-5)}{6!}\xi^{6-2\sigma}+\cdots\Big\}.
\end{align*}
We observe that the power series $\{\cdots\}(\xi)$ is absolutely convergent and it is increasing on the interval $(0,1)$. Moreover, we have
\begin{equation}\label{elementary estimate for sigma}
0<2^{2\sigma-1}-1<1,
\end{equation}
since $2^{2\cdot\frac{1}{2}-1}-1=0$, $2^{2\cdot1-1}-1=1$ and $\frac{d}{d\sigma}(2^{2\sigma-1}-1)=2^{2\sigma}\ln 2>0$. Therefore, we obtain 
$$g(\xi)=\{\cdots\}(\xi)|\xi|^{2\sigma}\leq \{\cdots\}(1)|\xi|^{2\sigma}=g(1)|\xi|^{2\sigma}=(2^{2\sigma-1}-1)|\xi|^{2\sigma}<|\xi|^{2\sigma}.$$
On the other hand, if $\xi>1$, we write
\begin{align*}
g(\xi)&=\frac{1}{2}|\xi|^{2\sigma}\Big\{|1+\tfrac{1}{\xi}|^{2\sigma}+|1-\tfrac{1}{\xi}|^{2\sigma}\Big\}-1\\
&=|\xi|^{2\sigma}-1+\Big\{\tfrac{2\sigma(2\sigma-1)}{2!}\tfrac{1}{\xi^2}+\tfrac{2\sigma(2\sigma-1)(2\sigma-2)(2\sigma-3)}{4!}\tfrac{1}{\xi^4}\\
&\quad\quad\quad\quad\quad\quad\quad\quad+\tfrac{2\sigma(2\sigma-1)(2\sigma-2)(2\sigma-3)(2\sigma-4)(2\sigma-5)}{6!}\tfrac{1}{\xi^6}+\cdots\Big\}.
\end{align*}
Then, the Laurent series $\{\cdots\}(\xi)$ is absolutely convergent, and it is decreasing on $(1,+\infty)$. Thus, we have
\begin{align*}
g(\xi)&=|\xi|^{2\sigma}-1+\{\cdots\}(\xi)\\
&\leq|\xi|^{2\sigma}-1+\{\cdots\}(1)=|\xi|^{2\sigma}-1+g(1)\\
&=|\xi|^{2\sigma}-1+(2^{2\sigma-1}-1)=|\xi|^{2\sigma}-(2-2^{2\sigma-1})\\
&<|\xi|^{2\sigma}\quad\textup{(by \eqref{elementary estimate for sigma})}.
\end{align*}
By the claim \eqref{claim for w}, we obtain
\begin{align*}
W_1(Q_{(\sigma)})= c^{1/2\sigma}\frac{\|Q_{(\sigma)}\|_{L^2}^{\frac{4\sigma-1}{\sigma}}\||\nabla|^\sigma Q_{(\sigma)}\|_{L^2}^{\frac{1}{\sigma}}}{\|Q_{(\sigma)}\|_{L^4}^4}.
\end{align*}
Moreover, by \eqref{claim for w} and the Pohozaev identities (Lemma \ref{Pohozaev}), 
\begin{align*}
W_2(Q_{(\sigma)})&= c^{1/2} \alpha\frac{\|Q_{(\sigma)}\|_{L^2}^3\||\nabla|^\sigma Q_{(\sigma)}\|_{L^2}}{\|Q_{(\sigma)}\|_{L^4}^4}\\
&= c^{1/2} (4\sigma-1)^{\frac{1-\sigma}{2\sigma}} \alpha \frac{\|Q_{(\sigma)}\|_{L^2}^{\frac{4\sigma-1}{\sigma}}\||\nabla|^\sigma Q_{(\sigma)}\|_{L^2}^{\frac{1}{\sigma}}}{\|Q_{(\sigma)}\|_{L^4}^4},
\end{align*}
where $\alpha=\alpha(\sigma)=1/\sqrt{\sigma(2\sigma-1)}$. Since $W_3(u)=W_1(u)^{1-\theta}W_2(u)^\theta$, collecting all, we get
$$\mathcal{W}(Q_{(\sigma)})=h(\theta)\cdot\frac{\|Q_{(\sigma)}\|_{L^2}^{\frac{4\sigma-1}{\sigma}}\||\nabla|^\sigma Q_{(\sigma)}\|_{L^2}^{\frac{1}{\sigma}}}{\|Q_{(\sigma)}\|_{L^4}^4},$$
where
$$h(\theta)=\Big(c^{1/2\sigma}+c^{1/2} (4\sigma-1)^{\frac{1-\sigma}{2\sigma}}\alpha-\big(c^{1/2\sigma}\big)^{1-\theta}\big(c^{1/2} (4\sigma-1)^{\frac{1-\sigma}{2\sigma}}\alpha\big)^\theta\Big).$$
Observe that $h(1)=c^{1/2\sigma}<1$. Hence, by continuity of $h(\theta)$, there exists $\theta\in(0,1)$, depending on $\sigma$, such that $h(\theta)<1$. This contradicts the \eqref{lambda to zero}.

\subsection{Dichotomy}
Now, we show that a dichotomy does not occur, in other words, $R_n\to0$ in $H^\sigma$. For contradiction, we assume that $R_n$ is not negligible in $H^\sigma$ as $n\to\infty$ so that by \eqref{Properties of two profiles}, either
\begin{equation}\label{phi1}
0<\lim_{n\to\infty}\frac{\|\phi\|_{L^2}^2}{\|v_n\|_{L^2}^2}=1-\lim_{n\to\infty}\frac{\|R_n\|_{L^2}^2}{\|v_n\|_{L^2}^2}<1
\end{equation}
or
\begin{equation}\label{phi2}
0<\lim_{n\to\infty} \frac{\|\mathcal{P}_{\lambda_n}^{1/2}\phi\|_{L^2}^2}{\|\mathcal{P}_{\lambda_n}^{1/2}v_n\|_{L^2}^2}=1-\lim_{n\to\infty}\frac{\|\mathcal{P}_{\lambda_n}^{1/2}R_n\|_{L^2}^2}{\|\mathcal{P}_{\lambda_n}^{1/2}v_n\|_{L^2}^2}<1
\end{equation}
holds up to a subsequence. Suppose that 
\begin{equation}\label{Step 2 assumption}
\lim_{n\to\infty}\frac{\|\phi\|_{L^2}^2}{\|v_n\|_{L^2}^2}\leq\lim_{n\to\infty}\frac{\|\mathcal{P}_{\lambda_n}^{1/2}\phi\|_{L^2}^2}{\|\mathcal{P}_{\lambda_n}^{1/2}v_n\|_{L^2}^2}.
\end{equation}
Define $f_n(t)=\mathcal{W}(t\phi(\frac{\cdot}{\lambda_n})+R_n(\frac{\cdot}{\lambda_n}))$. We will show that $f_n'(1)\geq\delta>0$ for sufficiently large $n$. Then, there exists $t\in(0,1)$ such that $\mathcal{W}(t\phi(\frac{\cdot}{\lambda_n})+R_n(\frac{\cdot}{\lambda_n}))\leq\frac{1}{c_{GN}}-\delta(1-t)+o_n(1)$, which contradicts the minimality of $\{u_n\}_{n=1}^\infty$. If \eqref{Step 2 assumption} does not hold, by \eqref{phi1} and \eqref{phi2}, we have
$$\lim_{n\to\infty}\frac{\|R_n\|_{L^2}^2}{\|v_n\|_{L^2}^2}\leq\lim_{n\to\infty}\frac{\|\mathcal{P}_{\lambda_n}^{1/2}R_n\|_{L^2}^2}{\|\mathcal{P}_{\lambda_n}^{1/2}v_n\|_{L^2}^2}.$$
Then, switching the roles of $\phi$ and $R_n$ in the definition of $f_n(t)$, i.e., $\mathcal{W}(\phi(\frac{\cdot}{\lambda_n})+tR_n(\frac{\cdot}{\lambda_n}))$, we can deduce a contradiction by the same argument. 

We compute $f_n'(1)$ using asymptotic orthogonality \eqref{Properties of two profiles} as in \eqref{static derivative calc}:
\begin{align*}
f_n'(1)&=W_1(u_n)\Big\{\tfrac{4\sigma-1}{\sigma}\tfrac{\|\phi\|_{L^2}^2}{\|v_n\|_{L^2}^2}+\tfrac{1}{\sigma}\tfrac{\|\mathcal{P}_{\lambda_n}^{1/2}\phi\|_{L^2}^2}{\|\mathcal{P}_{\lambda_n}^{1/2}v_n\|_{L^2}^2}-4\tfrac{\|\phi\|_{L^4}^4}{\|v_n\|_{L^4}^4}\Big\}\\
&+W_2(u_n)\Big\{3\tfrac{\|\phi\|_{L^2}^2}{\|v_n\|_{L^2}^2}+\tfrac{\|\mathcal{P}_{\lambda_n}^{1/2}\phi\|_{L^2}^2}{\|\mathcal{P}_{\lambda_n}^{1/2}v_n\|_{L^2}^2}-4\tfrac{\|\phi\|_{L^4}^4}{\|v_n\|_{L^4}^4}\Big\}\\
&-W_3(u_n)\Big\{\big(\tfrac{(4\sigma-1)(1-\theta)}{\sigma}+3\theta\big)\tfrac{\|\phi\|_{L^2}^2}{\|v_n\|_{L^2}^2} +\big(\tfrac{1-\theta}{\sigma}+\theta\big)\tfrac{\|\mathcal{P}_{\lambda_n}^{1/2}\phi\|_{L^2}^2}{\|\mathcal{P}_{\lambda_n}^{1/2}v_n\|_{L^2}^2}-4\tfrac{\|\phi\|_{L^4}^4}{\|v_n\|_{L^4}^4}\Big\}+o_n(1).
\end{align*}
By the trivial estimate $W_3=W_1^{1-\theta}W_2^\theta\leq(1-\theta)W_1+\theta W_2$,
\begin{align*}
f_n'(1)&\geq\theta W_1\Big\{(\tfrac{5\sigma-2}{\sigma}+\tfrac{1-\sigma}{\sigma}\theta)\tfrac{\|\phi\|_{L^2}^2}{\|v_n\|_{L^2}^2}+(\tfrac{2-\sigma}{\sigma}-\tfrac{1-\sigma}{\sigma}\theta)\tfrac{\|\mathcal{P}_{\lambda_n}^{1/2}\phi\|_{L^2}^2}{\|\mathcal{P}_{\lambda_n}^{1/2}v_n\|_{L^2}^2}-4\tfrac{\|\phi\|_{L^4}^4}{\|v_n\|_{L^4}^4}\Big\}\\
&+(1-\theta)W_2\Big\{(3+\tfrac{1-\sigma}{\sigma}\theta)\tfrac{\|\phi\|_{L^2}^2}{\|v_n\|_{L^2}^2}+(1-\tfrac{1-\sigma}{\sigma}\theta)\tfrac{\|\mathcal{P}_{\lambda_n}^{1/2}\phi\|_{L^2}^2}{\|\mathcal{P}_{\lambda_n}^{1/2}v_n\|_{L^2}^2}-4\tfrac{\|\phi\|_{L^4}^4}{\|v_n\|_{L^4}^4}\Big\}+o_n(1).
\end{align*}
For notational convenience, let
$$A=\tfrac{\|\phi\|_{L^2}^2}{\|v_n\|_{L^2}^2},\ B=\tfrac{\|\mathcal{P}_{\lambda_n}^{1/2}\phi\|_{L^2}^2}{\|\mathcal{P}_{\lambda_n}^{1/2}v_n\|_{L^2}^2},\ C=\tfrac{\|\phi\|_{L^4}^4}{\|v_n\|_{L^4}^4},\ W_1=W_1(u_n),\ W_2=W_2(u_n),\ W_3=W_3(u_n).$$
Choosing $\theta<1$ sufficiently close to $1$, we write
$$f_n'(1)\geq\theta W_1\Big\{\tfrac{4\sigma-1}{\sigma}^-A+\tfrac{1}{\sigma}^+B-4C\Big\}+(1-\theta)W_2\Big\{\tfrac{1+2\sigma}{\sigma}^-A+\tfrac{2\sigma-1}{\sigma}^+B-4C\Big\}+o_n(1),$$
where $a^+$ is a preselected number that is greater than $a$ but sufficiently close to $a$ depending only on $\theta$, and similarly for $b^-$. Here, $\tfrac{4\sigma-1}{\sigma}^-+\tfrac{1}{\sigma}^+=4$ and $\tfrac{1+2\sigma}{\sigma}^-+\tfrac{2\sigma-1}{\sigma}^+=4$.

We will estimate $C=\tfrac{\|\phi\|_{L^4}^4}{\|v_n\|_{L^4}^4}$ using the following lemma.
\begin{lemma}\label{elementary lemma}
For $\theta\in(0,1)$, let $\ell(s)=1+s-s^\theta$.\\
$(i)$ If $\theta^{\frac{1}{1-\theta}}\leq s_1\leq s_2$, then $\ell(s_1)\leq \ell(s_2)$.\\
$(ii)$ If $0\leq s_1\leq\theta^{\frac{1}{1-\theta}}$ and $s_1\leq s_2$, then $\ell(s_1)\leq (1-\theta^{\frac{\theta}{\theta-1}}(1-\theta))^{-1}\ell(s_2)$.
\end{lemma}

\begin{proof}
Observe that $\ell$ is convex on $[0,+\infty)$, $\ell(0)=1$ and $\ell(s)$ has a critical number $\theta^{\frac{1}{1-\theta}}$, where $\ell$ has an absolute minimum $\ell(\theta^{\frac{1}{1-\theta}})=1-\theta^{\frac{\theta}{\theta-1}}(1-\theta)$.  The lemma then follows.
\end{proof}

\begin{remark}
$\theta^{\frac{1}{1-\theta}}\to1/e$ and $(1-\theta^{\frac{\theta}{\theta-1}}(1-\theta))^{-1}\to 1$ as $\theta\to 1-$.
\end{remark}

Let
\begin{align*}
x&=\|\phi\|_{L^2}^{\frac{4\sigma-1}{\sigma}}\|\mathcal{P}_{\lambda_n}^{1/2} \phi\|_{L^2}^{\frac{1}{\sigma}},&& y= \alpha\|\phi\|_{L^2}^3\|\mathcal{P}_{\lambda_n}^{1/2} \phi\|_{L^2},\\
X&=\|v_n\|_{L^2}^{\frac{4\sigma-1}{\sigma}}\|\mathcal{P}_{\lambda_n}^{1/2} v_n\|_{L^2}^{\frac{1}{\sigma}},&& Y=\alpha\|v_n\|_{L^2}^3\|\mathcal{P}_{\lambda_n}^{1/2} v_n\|_{L^2}.
\end{align*}
By minimality of $\{u_n\}_{n=1}^\infty$ and scaling, we have
\begin{align*}
\frac{1}{c_{GN}}&=\mathcal{W}(u_n)+o_n(1)\\
&=W_1(v_n(\tfrac{\cdot}{\lambda_n}))+W_2(v_n(\tfrac{\cdot}{\lambda_n}))-W_1(v_n(\tfrac{\cdot}{\lambda_n}))^{1-\theta}W_2(v_n(\tfrac{\cdot}{\lambda_n}))^\theta+o_n(1)\\
&=\cdots=\frac{X+(\lambda_n^{1-\sigma}Y)-X^{1-\theta}(\lambda_n^{1-\sigma}Y)^\theta}{\|v_n\|_{L^4}^4}+o_n(1)
\end{align*}
and similarly, 
$$\frac{1}{c_{GN}}\leq \mathcal{W}(\phi(\tfrac{\cdot}{\lambda_n}))=\frac{x+(\lambda_n^{1-\sigma}y)-x^{1-\theta}(\lambda_n^{1-\sigma}y)^\theta}{\|\phi\|_{L^4}^4}.$$
Thus, we obtain 
\begin{align*}
C=\frac{\|\phi\|_{L^4}^4}{\|v_n\|_{L^4}^4}&\leq \frac{x+(\lambda_n^{1-\sigma}y)-x^{1-\theta}(\lambda_n^{1-\sigma}y)^\theta}{X+(\lambda_n^{1-\sigma}Y)-X^{1-\theta}(\lambda_n^{1-\sigma}Y)^\theta}+o_n(1)\\
&=\frac{x}{X}\frac{1+\frac{\lambda_n^{1-\sigma}y}{x}-(\frac{\lambda_n^{1-\sigma}y}{x})^\theta}{1+\frac{\lambda_n^{1-\sigma}Y}{X}-(\frac{\lambda_n^{1-\sigma}Y}{X})^\theta}+o_n(1)=\frac{x}{X}\frac{\ell(\frac{\lambda_n^{1-\sigma}y}{x})}{\ell(\frac{\lambda_n^{1-\sigma}Y}{X})}+o_n(1).
\end{align*}
By the assumption \eqref{Step 2 assumption}, we have $\frac{y}{Y}\leq\frac{x}{X}$ $(\Rightarrow\frac{y}{x}\leq\frac{Y}{X})$.

Suppose that $\lambda_n\to+\infty$ or $\lambda_n=1$ and $\theta^{\frac{1}{1-\theta}}\leq\frac{y}{x}$ (in both cases, $\theta^{\frac{1}{1-\theta}}\leq\frac{\lambda_n^{1-\sigma}y}{x}$ for large $n$). Then, by Lemma \ref{elementary lemma} ($\Rightarrow C\leq\frac{x}{X}=A^{\frac{4\sigma-1}{2\sigma}}B^{\frac{1}{2\sigma}}$) and Young's inequalities, we get
\begin{align*}
f_n'(1)&\geq\theta W_1\Big\{\tfrac{4\sigma-1}{\sigma}^-A+\tfrac{1}{\sigma}^+B-4\Big(\tfrac{4\sigma-1}{4\sigma}^-A^{2^+}+\tfrac{1}{4\sigma}^+B^{2^-}\Big)\Big\}\\
&+(1-\theta)W_2\Big\{\tfrac{1+2\sigma}{\sigma}^-A+\tfrac{2\sigma-1}{\sigma}^+B-4\Big(\tfrac{1+2\sigma}{4\sigma}^-A^{\frac{8\sigma-2}{1+2\sigma}^+}+\tfrac{2\sigma-1}{4\sigma}^+B^{\frac{2}{2\sigma-1}^-}\Big)\Big\}+o_n(1).
\end{align*}
Here, Young's inequalities are applied so that $\tfrac{4\sigma-1}{\sigma}^-=4\cdot \tfrac{4\sigma-1}{4\sigma}^-$, $\tfrac{1}{\sigma}^+=4\cdot \tfrac{1}{4\sigma}^+$, $\tfrac{1+2\sigma}{\sigma}^-=4\cdot \tfrac{1+2\sigma}{4\sigma}^-$ and $\tfrac{2\sigma-1}{\sigma}^+=4\cdot\tfrac{2\sigma-1}{4\sigma}^+$. Note that $\frac{8\sigma-2}{1+2\sigma}^+>1$ and $\frac{2}{2\sigma-1}^->1$, since $\sigma\in(\frac{1}{2},1)$. Therefore, by the assumption \eqref{Step 2 assumption} $(\Rightarrow 0<A<1\textup{ or }0<B<1)$, we conclude that $f_n'(1)>\delta>0$ for sufficiently large $n$.

Suppose that $\lambda_n=1$, $\frac{y}{x}\leq\theta^{\frac{1}{1-\theta}}$ and $\theta$ is sufficiently close to $1$. We claim that $A\leq (\alpha e^+)^{-\frac{2\sigma}{1-\sigma}}3^{2\sigma}$. By the assumption, we have
$$\frac{y}{x}=\tfrac{\alpha\|\phi\|_{L^2}^3\|\mathcal{P}_1^{1/2}\phi\|_{L^2}}{\|\phi\|_{L^2}^{\frac{4\sigma-1}{\sigma}}\|\mathcal{P}_1^{1/2}\phi\|_{L^2}^{\frac{1}{\sigma}}}=\alpha \Big\{\tfrac{\|\phi\|_{L^2}^2}{\|\mathcal{P}_1^{1/2}\phi\|_{L^2}^2}\Big\}^{\frac{1-\sigma}{2\sigma}}\leq \theta^{\frac{1}{1-\theta}}=\tfrac{1}{e^+}.$$
Hence, $A=\tfrac{\|\phi\|_{L^2}^2}{\|v_n\|_{L^2}^2}=\|\phi\|_{L^2}^2\leq (\alpha e^+)^{-\frac{2\sigma}{1-\sigma}}\|\mathcal{P}_1^{1/2}\phi\|_{L^2}^2$. Moreover, we have
\begin{equation}\label{boundedness of P1}
\|\mathcal{P}_1^{1/2}f\|_{L^2}^2\leq 3^{2\sigma}\||\nabla|^\sigma f\|_{L^2}^2,
\end{equation}
equivalently, $p_1(\xi)\leq 3^{2\sigma}|\xi|^{2\sigma}$. Indeed, if $\xi\geq\frac{1}{2}$, then $|\xi+1|\leq |\xi|+1\leq 3|\xi|$, so $p_1(\xi)\leq3^{2\sigma}|\xi|^{2\sigma}-1-2\sigma\xi\leq 3^{2\sigma}|\xi|^{2\sigma}$. If $\xi\leq-\frac{1}{2}$, then $|\xi+1|\leq|\xi|$, so  $p_1(\xi)\leq|\xi|^{2\sigma}-1-2\sigma\xi\leq (1+2\sigma\cdot 2^{2\sigma-1})|\xi|^{2\sigma}\leq3^{2\sigma}|\xi|^{2\sigma}$. If $|\xi|\leq\frac{1}{2}$, then by the mean value theorem, $p_1(\xi)\leq 1+2\sigma\xi+\sigma(2\sigma-1)2^{2-2\sigma}\xi^2-1-2\sigma\xi\leq \sigma(2\sigma-1)|\xi|^{2\sigma}$. Thus, by \eqref{boundedness of P1} with $\|v_n\|_{L^2}^2=\|\nabla v_n\|_{L^2}^2=1$, we prove that 
$$A\leq (\alpha e)^{-\frac{2\sigma}{1-\sigma}}3^{2\sigma}\||\nabla|^\sigma\phi\|_{L^2}^2\leq (\alpha e^+)^{-\frac{2\sigma}{1-\sigma}}3^{2\sigma}\||\nabla|^\sigma v_n\|_{L^2}^2\leq (\alpha e^+)^{-\frac{2\sigma}{1-\sigma}}3^{2\sigma}.$$
Finally, by Lemma \ref{elementary lemma} ($\Rightarrow C\leq\frac{x}{X}=1^+\cdot A^{\frac{4\sigma-1}{2\sigma}}B^{\frac{1}{2\sigma}})$, Young's inequalities as above and the claim, we get
\begin{align*}
f_n'(1)&\geq\theta W_1\Big\{\tfrac{4\sigma-1}{\sigma}^-A+\tfrac{1}{\sigma}^+B-4^+\Big(\tfrac{4\sigma-1}{4\sigma}A^{2}+\tfrac{1}{4\sigma}B^{2}\Big)\Big\}\\
&+(1-\theta)W_2\Big\{\tfrac{1+2\sigma}{\sigma}^-A+\tfrac{2\sigma-1}{\sigma}^+B-4^+\Big(\tfrac{1+2\sigma}{4\sigma}A^{\frac{8\sigma-2}{1+2\sigma}}+\tfrac{2\sigma-1}{4\sigma}B^{\frac{2}{2\sigma-1}}\Big)\Big\}+o_n(1)\\
&\geq\delta>0,
\end{align*}
where $\tfrac{1}{\sigma}^+=4^+\cdot\tfrac{1}{\sigma}$ and $\tfrac{2\sigma-1}{\sigma}^+=4^+\cdot\tfrac{2\sigma-1}{\sigma}$. Here, we used the fact that $0<A<(\alpha e)^{-\frac{2\sigma}{1-\sigma}}3^{2\sigma}<1$.

\subsection{Spreading out} By Section 4.3, $\{\phi(\frac{\cdot}{\lambda_n})\}_{n=1}^\infty$ is also a minimizing sequence. Then it is enough to show that this minimizing sequence is not spread out $(\lambda_n\to+\infty)$, because if $\lambda_n=1$, then $\phi$ minimizes $\mathcal{W}(u)$.

For contradiction, we assume that $\lambda_n\to \infty$ so that $u_n$ is concentrated in low frequencies for large $n$. Since $\mathcal{P}_k\approx\sigma(2\sigma-1)(-\Delta)$ in low frequencies (Lemma A.2 $(i)$), we have
$$W_1(u_n)=\frac{\|\phi(\frac{\cdot}{\lambda_n})\|_{L^2}^{\frac{4\sigma-1}{\sigma}}\|\nabla \phi(\frac{\cdot}{\lambda_n})\|_{L^2}^{\frac{1}{\sigma}}}{\|\phi(\frac{\cdot}{\lambda_n})\|_{L^4}^4}+o_n(1)=\lambda_n^{1-\frac{1}{\sigma}}\frac{\|\phi\|_{L^2}^{\frac{4\sigma-1}{\sigma}}\|\nabla \phi\|_{L^2}^{\frac{1}{\sigma}}}{\|\phi\|_{L^4}^4}+o_n(1)\to0$$
and thus $W_3(u_n)=W_1(u_n)^{1-\theta}W_2(u_n)^\theta\to 0$ as $n\to 0$. Hence, for sufficiently large $n$, by scaling,
$$\mathcal{W}(u_n)=W_2(u_n)+o_n(1)=\frac{\|\phi(\frac{\cdot}{\lambda_n})\|_{L^2}^3\|\nabla \phi(\frac{\cdot}{\lambda_n})\|_{L^2}}{\|\phi(\frac{\cdot}{\lambda_n})\|_{L^4}^4}+o_n(1)=\frac{\|\phi\|_{L^2}^3\|\nabla \phi\|_{L^2}}{\|\phi\|_{L^4}^4}+o_n(1).$$
We note that $\phi$ should be contained in $H^1$ (not only in $H^\sigma$), because if $\|\phi\|_{H^1}=\infty$, then $\mathcal{W}(u_n)\to \infty$, which contradicts the minimality of $\{u_n\}_{n=1}^\infty$. Let $Q_{(1)}$ be the ground state for the nonlinear elliptic equation
$$-\Delta Q_{(1)}+Q_{(1)}-Q_{(1)}^3=0.$$
Then, it follows from the minimization property of $Q_{(1)}$ that
$$\mathcal{W}(u_n)\to\frac{\|Q_{(1)}\|_{L^2}^3\|\nabla Q_{(1)}\|_{L^2}}{\|Q_{(1)}\|_{L^4}^4}.$$
By the interpolation inequality 
$$\||\nabla|^\sigma Q_{(1)}\|_{L^2}^2\leq \|Q_{(1)}\|_{L^2}^{2-2\sigma}\|\nabla Q_{(1)}\|_{L^2}^{2\sigma}$$
and the minimization property of $Q_{(\sigma)}$, we get
$$\mathcal{W}(u_n)\geq\frac{\|Q_{(1)}\|_{L^2}^{\frac{4\sigma-1}{\sigma}}\||\nabla|^\sigma Q_{(1)}\|_{L^2}^{\frac{1}{\sigma}}}{\|Q_{(1)}\|_{L^4}^4}+o_n(1)>\frac{\|Q_{(\sigma)}\|_{L^2}^{\frac{4\sigma-1}{\sigma}}\||\nabla|^\sigma Q_{(\sigma)}\|_{L^2}^{\frac{1}{\sigma}}}{\|Q_{(\sigma)}\|_{L^4}^4}+\delta$$
for some $\delta>0$. However, it makes a contradiction to the minimality of $\{u_n\}_{n=1}^\infty$, since by \eqref{lambda to zero}, if $\lambda_n'\to 0$, then
$$\mathcal{W}(Q_{(\sigma)}(\tfrac{\cdot}{\lambda_n'}))\to\frac{\|Q_{(\sigma)}\|_{L^2}^{\frac{4\sigma-1}{\sigma}}\||\nabla|^\sigma Q_{(\sigma)}\|_{L^2}^{\frac{1}{\sigma}}}{\|Q_{(\sigma)}\|_{L^4}^4}.$$

\subsection{Regularity} The regularity follows from standard arguments. Since $u \in H^1(\mathbb R)$, then $u$ is continuous and even, by Morrey embedding, it is $C^{1/2}_{loc}(\mathbb R)$. Hence one has that 
$$
\mathcal P_1 u \in C^{1/2}(\mathbb R). 
$$
Then the result follows by standard elliptic regularity and bootstrap. 

\appendix

\section{Properties of the Fourier Multiplier operator $\mathcal{P}_k$}

We collect properties of the Fourier multiplier operator $\mathcal{P}_k$ defined by
$$(\mathcal{P}_kf)^\wedge(\xi)=p_k(\xi)\hat{f}(\xi),$$
where
$$p_k(\xi)=|\xi+k|^{2\sigma}-|k|^{2\sigma}-2\sigma|k|^{2\sigma-2}k\cdot\xi.$$

\begin{lemma}[Asymptotics of $p_k$]\label{asymptotics}
Suppose that $\sigma\in(\frac{1}{2},1)$ and $k\in\mathbb{R}$ with $k\neq 0$.\\
$(i)$ (Low frequencies)
$$p_k(\xi)=\sigma(2\sigma-1)|k|^{2\sigma-2}|\xi|^2\Big(1+O_{k,\sigma}(|\xi|)\Big)\quad\textup{ as }\xi\to0.$$
$(ii)$ (High frequencies)
$$p_k(\xi)=|\xi|^{2\sigma}\Big(1+O_{k,\sigma}\Big(\frac{1}{|\xi|}\Big)\Big)\quad\textup{ as }\xi\to\infty.$$
\end{lemma}

\begin{proof}
For $(i)$, we compute
\begin{equation}\label{p_k calculation}
\begin{aligned}
p_k(\xi)&=|\xi+k|^{2\sigma}-|k|^{2\sigma}-2\sigma|k|^{2\sigma-2}k\cdot\xi&&\Longrightarrow p_k(0)=0,\\
p_k'(\xi)&=2\sigma |\xi+k|^{2\sigma-2}(\xi+k)-2\sigma|k|^{2\sigma-2}k&&\Longrightarrow p_k'(0)=0,\\
p_k''(\xi)&=2\sigma(2\sigma-1)|\xi+k|^{2\sigma-2}&&\Longrightarrow p_k''(0)=2\sigma(2\sigma-1)|k|^{2\sigma-2},\\
p_k'''(\xi)&=2\sigma(2\sigma-1)(2\sigma-2)|\xi+k|^{2\sigma-3}(\xi+k)&&\Longrightarrow |p_k'''(\xi)|\lesssim_{k,\sigma}1\textup{ for }|\xi|\leq\tfrac{|k|}{2},
\end{aligned}
\end{equation}
where the implicit constant depends on $k$ and $\sigma$. Thus, by the third order Taylor series expansion, we prove $(i)$. For $(ii)$, we write 
$$\frac{p_k(\xi)}{|\xi|^{2\sigma}}=\Big|1+\frac{k}{|\xi|}\Big|^{2\sigma}-\Big(\frac{|k|}{|\xi|}\Big)^{2\sigma}-2\sigma\Big(\frac{|k|}{|\xi|}\Big)^{2\sigma-2} \frac{|k|}{|\xi|}\cdot\frac{\xi}{|\xi|}.$$
Then, using the third order Taylor series expansion for $\frac{p_k(\xi)}{|\xi|^{2\sigma}}$ with $\frac{k}{|\xi|}$ being small, we prove the high frequency asymptotics.
\end{proof}

Using Lemma \ref{asymptotics}, we prove that $\mathcal{P}_k\approx\sigma(2\sigma-1)|k|^{2\sigma-2}(-\Delta)$ in low frequencies, and $\mathcal{P}_k\approx (-\Delta)^\sigma$ in high frequencies in the following sense.

\begin{lemma}[Limiting behavior of $\mathcal{P}_k$] Suppose that $\sigma\in(\frac{1}{2},1)$ and $k\in\mathbb{R}$ with $k\neq 0$.\\
$(i)$ (Low frequencies) For any $\phi\in H^1$, 
$$\int_{\mathbb{R}}\big(\mathcal{P}_k\phi(\tfrac{x}{\lambda_n})\big)\overline{\phi(\tfrac{x}{\lambda_n})}dx-\sigma(2\sigma-1)|k|^{2\sigma-2}\int_{\mathbb{R}}((-\Delta)\phi(\tfrac{x}{\lambda_n}))\overline{\phi(\tfrac{x}{\lambda_n})}dx\to 0\quad\textup{ as }\lambda_n\to\infty.$$
$(ii)$ (High frequencies) For any $\phi\in H^\sigma$, 
$$\int_{\mathbb{R}}\big(\mathcal{P}_k\phi(\tfrac{x}{\lambda_n})\big)\overline{\phi(\tfrac{x}{\lambda_n})}dx-\int_{\mathbb{R}}((-\Delta)^\sigma\phi(\tfrac{x}{\lambda_n}))\overline{\phi(\tfrac{x}{\lambda_n})}dx\to0\quad\textup{ as }\lambda_n\to0.$$
\end{lemma}

\begin{proof}
By Plancherel theorem, change of the variables, Lemma \ref{asymptotics} $(i)$ and the dominate convergence theorem, 
\begin{align*}
\int_{\mathbb{R}}\big(\mathcal{P}_k\phi(\tfrac{x}{\lambda_n})\big)\overline{\phi(\tfrac{x}{\lambda_n})}dx&=\int_{\mathbb{R}}(|\xi+k|^{2\sigma}-|k|^{2\sigma}-2\sigma|k|^{2\sigma-2}k\cdot\xi)|\lambda_n\hat{\phi}(\lambda_n\xi)|^2 d\xi\\
&=\lambda_n\int_{\mathbb{R}}(|\tfrac{\xi}{\lambda_n}+k|^{2\sigma}-|k|^{2\sigma}-2\sigma|k|^{2\sigma-2}k\cdot\tfrac{\xi}{\lambda_n})|\hat{\phi}(\xi)|^2 d\xi\\
&=\lambda_n\int_{\mathbb{R}}\sigma(2\sigma-1)|k|^{2\sigma-2}|\tfrac{\xi}{\lambda_n}|^2|\hat{\phi}(\xi)|^2 d\xi+o_n(1)\\
&=\int_{\mathbb{R}}\sigma(2\sigma-1)|k|^{2\sigma-2}|\xi|^2|\lambda_n\hat{\phi}(\lambda_n\xi)|^2 d\xi+o_n(1)\\
&=\sigma(2\sigma-1)|k|^{2\sigma-2}\int_{\mathbb{R}}((-\Delta)\phi(\tfrac{x}{\lambda_n}))\overline{\phi(\tfrac{x}{\lambda_n})}dx+o_n(1).
\end{align*}
Similarly, we prove $(ii)$.
\end{proof}

In all frequencies, $\mathcal{P}_k$ is bounded by $(-\Delta)^\sigma$ up to constant multiple.
\begin{lemma}[Bounds for $p_k$]\label{p_k bound}
Let $\sigma\in(\frac{1}{2},1)$ and $k\in\mathbb{R}$ with $k\neq 0$. Then, $p_k$ satisfies the following bounds whose implicit constants do not depend on $k$.\\
$(i)$ (Low frequencies) $p_k(\xi)\sim 2\sigma(2\sigma-1)|k|^{2\sigma-2}|\xi|^2$ for $|\xi|\leq\frac{|k|}{2}$.\\
$(ii)$ (High frequencies) $p_k(\xi)\sim|\xi|^{2\sigma}$ for $|\xi|\geq\frac{|k|}{2}$.\\
$(iii)$ (Upper bound) $p_k(\xi)\lesssim|\xi|^{2\sigma}$ for all $\xi$.
\end{lemma}

\begin{proof}
We observe from \eqref{p_k calculation} that $p_k(\xi)$ is a strictly convex function having only one critical number $0$ with the absolute minimum $p_k(0)=0$.\\
$(i)$ Suppose that $|\xi|\leq\frac{|k|}{2}$. Then, by \eqref{p_k calculation}, $|p_k''(\xi)|\sim 2\sigma(2\sigma-1)|k|^{2\sigma-2}$. Thus, by the second order Taylor expansion, we obtain
\begin{align*}
p_k(\xi)&\sim p_k(0)+p_k'(0)\xi+p_k''(\xi_*)|\xi|^2\quad\textup{(for some $\xi_*\in(0,\xi)$)}\\
&\sim 2\sigma(2\sigma-1)|k|^{2\sigma-2}|\xi|^2.
\end{align*}
$(ii)$ If $|\xi|\geq2|k|$, by the second order Taylor series expansion as in $(i)$, we can prove that
$$p_k(\xi)=|\xi|^{2\sigma}\Big\{|1+\tfrac{k}{|\xi|}|^{2\sigma}-(\tfrac{|k|}{|\xi|})^{2\sigma}-2\sigma\tfrac{\xi}{|k|}\Big\}\sim|\xi|^{2\sigma}.$$
Suppose that $\frac{|k|}{2}\leq|\xi|\leq 2|k|$. Then,
$$\min_{\frac{|k|}{2}\leq|\xi|\leq 2|k|}p_k(\xi)=\min_{\xi=\pm\frac{|k|}{2}}p_k(\xi)=\min\Big\{(\tfrac{3}{2})^{2\sigma}-1-\sigma, (\tfrac{1}{2})^{2\sigma}-1+\sigma\Big\}|k|^{2\sigma}\geq|\xi|^{2\sigma}.$$
Thus, $p_k(\xi)\gtrsim |\xi|^{2\sigma}$. Similarly, we can show that $\max_{\frac{|k|}{2}\leq|\xi|\leq 2|k|}p_k(\xi)\sim|k|^{2\sigma}$, which implies that $p_k(\xi)\lesssim|\xi|^{2\sigma}.$\\
$(iii)$ follows immediately from $(i)$ and $(ii)$.
\end{proof}

\section{Pohozaev Identities}

\begin{lemma}[Pohozaev identities]\label{Pohozaev} Let $Q_{(\sigma)}$ is a solution to the nonlinear elliptic equation $(-\Delta)^\sigma Q_{(\sigma)} +Q_{(\sigma)}-Q_{(\sigma)}^3=0$. Then,
$$\|Q_{(\sigma)}\|_{\dot{H}^\sigma}^2=\frac{1}{4\sigma-1}\|Q_{(\sigma)}\|_{L^2}^2,\ \|Q_{(\sigma)}\|_{L^{4}}^{4}=\frac{4\sigma}{4\sigma-1}\|Q_{(\sigma)}\|_{L^2}^2.$$
\end{lemma}

\begin{proof}
Multiplying the equation by $Q_{(\sigma)}$ (and $x\cdot\nabla Q_{(\sigma)}$), integrating over $\mathbb{R}$ and then applying integration by parts, we obtain
$$\|Q_{(\sigma)}\|_{\dot{H}^\sigma}^2+\|Q_{(\sigma)}\|_{L^2}^2-\|Q_{(\sigma)}\|_{L^{4}}^{4}=0,\ -(\tfrac{1}{2}-\sigma)\|Q_{(\sigma)}\|_{\dot{H}^\sigma}^2-\tfrac{1}{2}\|Q_{(\sigma)}\|_{L^2}^2+\tfrac{1}{4}\|Q_{(\sigma)}\|_{L^{4}}^{4}=0.$$
Here, for the second identity, we use $(-\Delta)^\sigma (x f)=x (-\Delta)^\alpha f-2\sigma(-\Delta)^{2\sigma-2}\nabla f$. Solving the equation for $\|Q_{(\sigma)}\|_{\dot{H}^\sigma}^2$ and $\|Q_{(\sigma)}\|_{L^{4}}^{4}$, we prove the lemma.
\end{proof}

\section*{Acknowledgements}
Y.H. would like to thank IH\'ES for their hospitality and support while he visited in the summer of 2014. Y.S. would like to thank the hospitality of the Department of Mathematics at University of Texas at Austin where part of the work was initiated. Y.S. acknowledges the support of ANR grants "HAB" and "NONLOCAL".  

\bibliographystyle{alpha} 
\bibliography{biblio}
\end{document}